\documentclass[a4paper,11pt]{amsart}

\usepackage{header}

\begin{document}

\title{Relations among $\IP$-twists}
\author[A. Hochenegger]{Andreas Hochenegger}
\address{ 
Dipartimento di Matematica ``Francesco Brioschi'', 
Politecnico di Milano, 
via Bonardi 9, 
20133 Milano  
} 
\email{andreas.hochenegger@polimi.it} 
\author[A. Krug]{Andreas Krug}
\address{
Institut f\"ur algebraische Geometrie,
Gottfried Wilhelm Leibniz Universit\"at Hannover,
Welfengarten 1,
30167 Hannover
}
\email{krug@math.uni-hannover.de}

\maketitle

\begin{abstract}
Given two $\IP$-objects in some algebraic triangulated category,
we investigate the possible relations among the associated $\IP$-twists.
The main result is that, under certain technical assumptions, the $\IP$-twists commute if and only if the $\IP$-objects are orthogonal. Otherwise, there are no relations at all.
In particular, this applies to most of the known pairs of $\IP$-objects on hyperk\"ahler varieties.
In order to show this, we relate $\IP$-twists to spherical twists and apply known results about the absence of relations between pairs of spherical twists.
\end{abstract}

\tableofcontents

\addtocontents{toc}{\protect\setcounter{tocdepth}{1}} 
  
\section{Introduction}

Spherical objects and the associated spherical twists are well-studied, 
in particular, one knows a lot about the relations between spherical twists.
Much is already present in \cite{ST} by P. Seidel and R. Thomas, where spherical objects were introduced in algebraic geometry.
Important geometric examples of such objects are line bundles in the derived category of a K3 surface, or, more generally, of a strict Calabi--Yau variety of arbitrary dimension.

Spherical objects and their twists generalise quite naturally to $\IP$-objects and their $\IP$-twists, see \cite{HT06} by D. Huybrechts and R. Thomas.
Here, the easiest examples are line bundles in the derived category of a hyperk\"ahler variety.
Most of the story about spherical objects carries over to $\IP$-objects, somewhat in analogy to the passage from K3 surfaces to higher dimensional hyperk\"ahler varieties.

We want to highlight a connection between these two notions, which is central for this article.
Given a hyperk\"ahler variety $X$, there is the inclusion $j \colon X \to \cX$ into its twistor space. The twistor space $\cX$ is the total space of all possible complex structures that can be put on $X$, which is an analytic space over $\IP^1$.

By \cite{HT06}, the (derived) pushforward $j_* \colon \cD(X) \to \cD(\cX)$ turns any $\IP$-object $P \in \cD(X)$ into a spherical object in $\cD(\cX)$, under an assumption, which can be interpreted as $P$ does not deform from $X$ to $\cX$.
Moreover, the associated twists fit together in the sense that the following diagram commutes
\[
\begin{tikzcd}
\cD(X) \ar[r, "j_*"] \ar[d, "\PPP_P"'] & \cD(\cX) \ar[d, "\TTT_{j_* P}"] \\
\cD(X) \ar[r, "j_*"] & \cD(\cX) 
\end{tikzcd}
\]
where $\PPP_P$ is the $\IP$-twist and $\TTT_{j_* P}$ is the spherical twist associated to $j_*P$.

This functor $j_* \colon \cD(X) \to \cD(\cX)$ is the prototype of what we call a \emph{spherification functor}, a functor that turns $\IP$-objects into spherical objects; see \autoref{sec:spherification:definition} for details on our definition. 
We establish the existence of such functors that turn $\IP$-objects into spherical objects also more abstractly, using our previous work on formality of $\IP$-objects \cite{Formality}:

\begin{theoremalpha}[\autoref{cor:formal:spherification}]
\label{main:spherification}
Let $P_1,\ldots,P_m$ be $\IP^n[k]$-objects in some algebraic triangulated category,
such that $k\ge 2$ is even, $n\ge 2$, $\gcd(k,nk/2)>1$ and $\Hom^*(P_i,P_j)$ is concentrated in degree $nk/2$ for $i\neq j$ (or zero).
Then there is a spherification functor $\FF \colon  \genby{P_1,\ldots,P_m} \to \cT$.
\end{theoremalpha}

We conjecture that the assumptions of \autoref{main:spherification} are mainly of technical nature, that is, we suppose that abstract spherification functors exist in greater generality.

Anyway, with a spherification functor at hand (of geometric origin or an abstract one), we can use the compatibility of $\IP$-twists and spherical twists. The relations between the spherical twists of two spherical objects are known by work of Y. Volkov \cite{Volkov}, which we pull back to $\IP$-twists along a spherification functor to obtain the following theorem. We say that two objects $P_1$ and $P_2$ of a triangulated category are \emph{isomorphic up to shift} if there exists some $m\in \IZ$ such that $P_1\cong P_2[m]$. Note that, if two $\IP$-objects are isomorphic up to shift, their associated twists are the same; see e.g.\ \cite[Lem.\ 2.4(ii)]{Krug-HK}.  

\begin{theoremalpha}[\autoref{thm:main}]
\label{main:relations}
Let $P_1,P_2$ be two $\IP^n[k]$-objects with $n\ge 2$ and $k\neq0,1$  which are not isomorphic up to shift in some algebraic triangulated category $\cS$ with $\Hom^*(P_1,P_2) \neq0$.
If there is a spherification functor $\FF \colon \genby{P_1,P_2} \to \cT$,
then the associated $\IP$-twists $\PP_1$ and $\PP_2$ generate the free group $F_2$.
\end{theoremalpha}

Note that the assumption $n\ge 2$ is logically not neccesary, as for any pair of non-orthogonal $\IP^1[k]$-objects (i.e.\ $k$-spherical objects) the associated $\IP$-twists (which are then the squares of the spherical twists) have no relations -- independently of the existence of a spherification functor. This follows directly from the results of Volkov on the (absence of) relations between pairs of spherical twists; see \autoref{rem:n=1} for details.   

We remark, that if $\Hom^*(P_1,P_2)=0$, then $\genby{\PP_1,\PP_2} \cong \IZ^2$. This was already known in essence, for completeness we discuss this case in detail in \autoref{sec:orthogonal}.

\subsection*{Symplectic geometry}
Finally, we want to mention that spherical objects and $\IP$-objects appear not only in algebraic geometry.
Actually, spherical twists appeared first in symplectic geometry as Dehn twists about Lagrangian spheres. 
 It was the homological mirror symmetry conjecture by M. Kontsevich \cite{Kontsevich} that motivated the successful search for spherical objects in algebraic geometry \cite{ST}.

Similarly, $\IP$-twists are (generalised) Dehn twists about Lagrangian projective spaces $\IP^n$ (this statement seems to be in parts still only conjectural; see \cite{MW18} by C.Y. Mak and W. Wu for details).

In the symplectic set-up, results analogous to \autoref{main:relations} on the absence of relations between Dehn twists along Lagrangian projective spaces can be found in \cite{Torricelli} by B.C. Torricelli. The authors were not aware of \emph{loc.\ cit}.\ before putting the first version of the present paper on the arXiv.
Not only the results, but also the proofs in \cite{Torricelli} are somewhat analogous to ours. We use spherification functors to deduce our results from Volkov's result on the absence of relations between spherical twists, while Torricelli uses the Hopf correspondence to deduce her result from known results about relations among Dehn twists about Lagrangian spheres by A. Keating \cite{Keating}. 

\subsection*{Conventions}

By $\kk$ we denote some arbitrary base field.
All categories in this article will be additive and $\kk$-linear.

Whenever we speak of an \emph{algebraic triangulated category}, we actually mean a (Karoubi complete) \emph{Morita enhanced triangulated category} $\cT$, that is, a category equivalent to the subcategory $\cD^c(\cA)$ of compact objects in the derived category of left dg modules over some small dg category $\cA$; see \cite[\S 4.2]{ALdg} for details on how the notion of Morita enhanced triangulated categories compares to dg enhanced triangulated categories. We could also work with \emph{large Morita enhanced triangulated categories} instead, i.e.\ categories equivalent to the full derived category $\cD(\cA)$ of some dg category. However, our main result is about faithfulness of a categorical group action. If the action of the free group $F_2$ on $\cD^c(\cA)$ via two $\IP$-twists is faithful, the same action on the bigger category $\cD(\cA)$ is faithful as well. 

Whenever we speak of a functor $\FF\colon \cS\cong \cD^c(\cA)\to \cT\cong \cD^c(\cB)$, we always assume this to be induced by a bimodule, i.e.\ it is of the form $\FF\cong M\otimes \blank$, where $M$ is a $\cB$-perfect $\cB\hh\cA$-bimodule, and the tensor product is the derived tensor product (by considering derived categories of left dg modules, and accordingly $\cB\hh\cA$-bimodule instead of $\cA\hh\cB$-bimodules, we have the same order for tensor products as for the composition of functors).
This assumption is needed in the proof of \autoref{prop:commutativity} which will be the only place where we work on the level of the bimodules instead of the functors between triangulated categories.  

Note that bounded derived categories of coherent sheaves on smooth varieties are equivalent to $\cD^c(\cA)$, where $\cA$ can be chosen as a dg algebra, i.e.\ a dg category with one object; see \cite[\S 3.1]{Bon-vdB}. Furthermore, every Fourier--Mukai transform is induced by a bimodule under this identification; see \cite[Ex.\ 4.3 \& \S 5.2]{ALdg}.

We write distinguished triangles as $A \to B \to C$, supressing the degree increasing morphism $C \to A[1]$.

In a triangulated category, we write $\Hom^*(A,B)$ for the graded vector space of derived homomorphisms $\bigoplus_i \Hom(A,B[i])[-i]$. In the setting of dg-modules, we will use $\Hom^\bullet(A,B)$ to denote the complex of homomorphisms.

\subsection*{Acknowledgements}

We thank Ivan Smith for making us aware that, in the first version of this paper, we missed the relevant reference \cite{Torricelli} as well as the assumption that the objects in \autoref{main:relations} must not be isomorphic up to shift.
We also thank an anonymous referee for useful comments and suggestions, and in particular for pointing out a gap in the first version of the proof of \autoref{prop:commutativity}.
\\ The first named author was partially supported by the project PRIN 2022 \emph{Unirationality, Hilbert schemes, and singularities}. 

\section{Enter spherical and $\IP$-objects}

We recall some basic facts from \cite{ST} and \cite{HT06},
using the slightly generalised notations as introduced in \cite{HKP}, \cite{Kunits} and \cite{Formality}.

\begin{definition}
Let $P$ be an object in some ($\kk$-linear) triangulated category, and let $n,k,d$ be integers with $n>0$.

We say that $P$ is \emph{$\IP^n[k]$-like} if $\End^*(P) \cong \kk[t]/t^{n+1}$ with $\deg(t) = k$.
If $P[d]$ is a Serre dual of $P$, that is,
\[
\Hom^*(P,\blank) \cong \Hom^*(\blank,P[d])^\vee
\]
then we say that $P$ is a \emph{$d$-Calabi--Yau} object.
We will drop $d$ as a prefix, if it is clear from the context.

If $P$ is $\IP^n[k]$-like and Calabi--Yau (in which case necessarily $d=nk$), then we say that $P$ is a \emph{$\IP^n[k]$-object}.

If $n=1$, then a $\IP^1[k]$-like object, or $\IP^1[k]$-object, is better known as a \emph{spherelike} or \emph{spherical} object, respectively.

If $k=2$, then a $\IP^n[2]$-object is better known as a \emph{$\IP^n$-object}. We will drop the exponent $n$, if it is clear from the context.
\end{definition}

\begin{proposition}[\cite{ST}]\label{prop:ST}
Let $S$ be a spherical object in some algebraic triangulated category $\cT$.
Then the \emph{spherical twist} $\TTT_S$ defined by the triangle
\[
S \otimes \Hom^*(S,\blank) \xto{\mathrm{ev}} \id \to \TTT_S
\]
is an autoequivalence of $\cT$.
\end{proposition}

\begin{proposition}
\label{prop:HTtwist}
Let $P$ be a $\IP^n[k]$-object in some algebraic triangulated category $\cT$.
Then $P$ gives rise to an autoequivalence of $\cT$, the \emph{$\IP$-twist} $\PPP_P$, which is defined in following diagram, where the tilted rows are exact triangles and $\mathrm{ev}$ lifts to the dashed arrow:
\begin{equation}\label{eq:Ptwistfunctor}
\begin{tikzcd}[row sep=0.5ex,column sep={7em,between origins}]
P\otimes \Hom^*(P,\blank) [-k]
\ar[rd, "H =  \id \otimes t^* - t \otimes \id"] \\
&P\otimes \Hom^*(P,\blank)
\ar[rd] \ar[dd, "\mathrm{ev}"]\\
&& \Cone_H(\blank) \ar[ld, dashed]\\
& \id \ar[ld] \\
\hphantom{\Hom^*(P,\blank)\otimes} \PPP_P
\end{tikzcd}
\end{equation}
\end{proposition}

\begin{remark}\label{rem:bimodtwist}
The above statement is due to \cite[Prop.\ 2.6]{HT06} in the setting of derived categories of coherent sheaves on smooth projective varieties. There, Fourier-Mukai kernels are used to overcome the issue that cones are in general not functorial, which could lead to the diagram \eqref{eq:Ptwistfunctor} not being well-defined.
In our more general set-up, where $\cT=\cD^c(\cA)$ for some dg-category $\cA$, we can ensure the existence of the functorial cones, by taking the cones in the derived category of $\cA$-$\cA$ bimodules where \eqref{eq:Ptwistfunctor} corresponds to 
\begin{equation}\label{eq:Ptwistbimod}
\begin{tikzcd}[row sep=0.5ex,column sep={7em,between origins}]
P\otimes_\kk {}^{\cA}P [-k]
\ar[rd, "H =  \id \otimes t^* - t \otimes \id"] \\
&P\otimes_\kk {}^{\cA}P
\ar[rd] \ar[dd, "\mathrm{ev}"]\\
&& \Cone_H \ar[ld, dashed]\\
& \cA \ar[ld] \\
 X
\end{tikzcd}
\end{equation}
${}^\cA P=\Hom_{\cA}^\bullet(P,\cA)$ is the $\kk$-$\cA$-bimodule representing
the functor $\Hom^*(P,\blank)$ and $\cA$ stands for the diagonal $\cA$-$\cA$-bimodule. 
The dashed arrow making the diagram commutative can actually be shown to be unique; see \cite[Lem.\ 2.1]{HT06} for the analogue in the category of Fourier--Mukai kernels or the second half of the proof of \autoref{prop:commutativity} for a very similar computation.
The proper definition of the $\IP$-twist is then $\PPP_P=X\otimes\blank$, the functor induced by the bimodule $X$ which is the cone of the unique dashed arrow. 

In \cite{ALunique}, a more general statement is shown for $\IP$-functors in place of $\IP$-objects: there we have that, even though the lift (the dashed arrow) might not be unique in general, its cone is. 

Of course, the issue of functoriality of cones already exists in the definition of the spherical twist in \autoref{prop:ST}. So the proper definition of $\TT_S$ is as the tensor product with the bimodule $Y$ defined as the cone 
\[
S\otimes_\kk {}^\cA S\xrightarrow{\mathrm{ev}}\cA\to Y\,. 
\]
\end{remark}

\begin{remark}
Associated to a spherical object $S$, we have the spherical twist $\TTT_S$ and the $\IP$-twist $\PPP_S$ which are related by
\[
\TT_{\!S}^2 \cong \PPP_S\,,
\]
see \cite[Prop. 2.9]{HT06}.
\end{remark}

The notion of spherelike and spherical objects can be generalised to functors.
Before giving the definitions, we need some canonical triangles associated to a functor $\FF \colon \cS \to \cT$ between algebraic triangulated categories, which admits both adjoints $\LL$ and $\RR$.

The \emph{twist} $\TT$ and \emph{dual twist} $\TT'$ associated to $\FF$ are defined by the triangles (using unit and counit of the adjunction):
\[
\FF\RR \to \id_{\cT} \to \TT
\quad
\text{and}
\quad
\TT' \to \id_{\cT} \to \FF\LL \,.
\]
Note that $\TT'$ is the left adjoint of $\TT$.

Similarly, the \emph{cotwist} $\CC$ and \emph{dual cotwist} $\CC'$ associated to $\FF$ are defined by the triangles
\[
\CC \to \id_{\cS} \to \RR\FF
\quad
\text{and}
\quad
\LL\FF \to \id_{\cS} \to \CC' \,.
\]
Again, $\CC'$ is the left adjoint of $\CC$.

Additionally note that we have natural morphisms
\[
\phi \colon \RR \to \RR\FF\LL \to \CC\LL[1]
\quad
\text{and}
\quad
\psi \colon \LL\TT[-1] \to \LL\FF\RR \to \RR \,.
\]

\begin{definition}[\cite{ALdg}, \cite{HM}]
Let $\FF \colon \cS \to \cT$ be a functor between algebraic triangulated categories, which admits both adjoints $\LL$ and $\RR$.

If $\CC$ is an autoequivalence, then we call $\FF$ a \emph{spherelike functor}.

If additionally $\phi \colon \RR \to \CC\LL[1]$ is an isomorphism, then we call $\FF$ a \emph{spherical functor}.
\end{definition}

\begin{proposition}[{\cite[Thm. 1.1]{ALdg}}]
Let $\FF \colon \cS \to \cT$ be a functor between algebraic triangulated categories, which admits both adjoints $\LL$ and $\RR$.

If $\FF$ satisfies two of the following four conditions, then $\FF$ satisfies all four of them:
\begin{itemize}
\item $\CC$ is an autoequivalence;
\item $\TT$ is an autoequivalence;
\item $\phi \colon \RR \to \CC\LL[1]$ is an isomorphism;
\item $\psi \colon \LL\TT[-1] \to \RR$ is an isomorphism.
\end{itemize}
In particular, such an $\FF$ is a spherical functor.
\end{proposition}

The notion is motivated by the following example.

\begin{example}
Let $S$ be some object in an algebraic triangulated category $\cT$, which admits a Serre dual.
Then the functor $\FF_S = S\otimes \blank \colon \cD(\kk\hh\mod) \to \cT$ admits both adjoints.

Moreover, we have that $\FF_S$ is a spherelike/spherical functor, if and only if, $S$ is a spherelike/spherical object, respectively.
\end{example}

\begin{remark}
There is also the generalisation of $\IP$-objects to $\IP$-functors.
We do not give the definition here, as we will not use it.
We refer to \cite{Add} where the notion of a (split) $\IP$-functor was first introduced,
this was generalised to (possibly non-split) $\IP$-functors in \cite{AL19}.
\end{remark}

\section{From $\IP$-objects to spherical objects}

\subsection{Introducing spherification functors}
\label{sec:spherification:definition}

Recall that for us, a functor $\FF\colon \cS\cong \cD^c(\cA)\to \cT\cong \cD^c(\cB)$ of algebraic triangulated categories is always of the form $\FF\cong M\otimes \blank$ for some $\cB$-$\cA$-module $M$ which is $\cB$-perfect. The latter assumption is needed to ensure that $\FF$ sends compact objects to compact objects (see \cite[\S 2.1.6]{ALdg}), but also gives a right adjoint $\RR\cong {}^\cB M\otimes \blank$ of $\FF$, where ${}^\cB M=\RHom_{\cB}(M,\cB)$. If $M$ is also $\cA$ perfect, $\FF$ also has a left adjoint $\LL\cong {}^\cA M\otimes\blank$ where ${}^\cA M=\RHom_{\cA}(M,\cA)$; see \cite[Cor.\ 2.2]{ALdg}. In this case, we say that $\FF$ is \emph{representated by a bimodule which is perfect over both sides}.  

Recall that then there is the triangle
\[
\LL\FF \xto{\eps} \id \xto{\alpha} \CC'
\]
where $\CC'$ is the \emph{dual cotwist} of $\FF$.
In the following, the morphism $\alpha \colon \id \to \CC'$ will be of importance.

\begin{definition}\label{def:spherification}
Let $\FF \colon \cS \to \cT$ be a functor of algebraic triangulated categories, which is represented by a bimodule which is perfect over both sides.
We say that $\FF$ is a \emph{weak spherification functor} for a $\IP^n[k]$-like object $P$ if 
\begin{itemize}
\item $\CC' P \cong P[k]$ and
\item the natural morphism $\alpha_P \colon P \to P[k]$ is a generator of $\Hom^*(P,P)$.
\end{itemize}
\end{definition}

\begin{remark}\label{rem:alphagen}
Since $P$ is a $\IP^n$-object, $\alpha_P$ being a generator of $\Hom^*(P,P)$ implies that $\Hom^*(P,P) = \kk[\alpha_P]/\alpha_P^{n+1}$.
\end{remark}

\begin{remark}\label{rem:kneq0}
If $k\neq0$, it is sufficient to ask that $\alpha_P$ is non-zero. Then it generates $\Hom^*(P,P)$ for degree reasons.
\end{remark}

We justify the name by the following lemma.

\begin{lemma}
\label{lem:weak:spherification}
Let $\FF \colon \cS \to \cT$ be a functor of algebraic triangulated categories,  which is represented by a bimodule which is perfect over both sides.
Let $P$ be a $\IP^n[k]$-like object in $\cS$.

If $\FF \colon \cS \to \cT$ is a weak spherification functor for $P$,
then $\FF P$ is a $(nk+k-1)$-spherelike object.

Conversely, if $\FF P$ is a $(nk+k-1)$-spherelike object and $\CC'P \cong P[k]$, then $\FF$ is a weak spherification functor for $P$.
\end{lemma}

\begin{proof}
Let $\FF$ be a weak spherification functor for $P$.
Apply $\Hom^*(\blank,P)$ to the triangle $\LL\FF P \to P \to \CC'P \cong P[k]$ to get
\begin{equation}\label{eq:Hom:triangle}
\Hom^*(\FF P,\FF P) \from \Hom^*(P,P) \xfrom{\alpha_P^*} \Hom^*(P[k],P)\,.
\end{equation}
Under the isomorphism $\Hom^*(P,P) \cong \kk[\alpha_P]/\alpha_P^{n+1}$, the map $\alpha_P^*$ induces isomorphisms on the components $\kk\cdot \alpha_P^i[-k] \xto{\sim} \kk\cdot \alpha_P^{i+1}$ for $0 \leq i < n$. It follows that $\Hom^*(\FF P,\FF P) \cong \kk[s]/s^2$ with $\deg(s)=nk+k-1$.
In other words, $\FF P$ is a $(nk+k-1)$-spherelike object.

Now suppose that $\CC' P \cong P[k]$, and that $\FF P$ is a $(nk+k-1)$-spherelike object, which means that $\Hom^*(\FF P, \FF P)\cong \IC[s]/s^2$ with $\deg s=nk+k-1$.
Using triangle \eqref{eq:Hom:triangle}, this is only possible if $\alpha_P\colon P \to P[k]$ is non-zero. 
In the case $k\neq 0$ this already suffices to conclude that $\FF$ is a weak spherification functor by \autoref{rem:kneq0}. 

For $k=0$, we have to study $\alpha_P\in \Hom(P,P)\cong \kk[t]/t^{n+1}$ a bit closer. Note that $\alpha_P^*$ is just given by multiplication in $\kk[t]/t^{n+1}$ with a polynomial that we denote again by $\alpha_P$. 
Again by \eqref{eq:Hom:triangle} this multiplication map must have a one-dimensional kernel. Hence 
\[
 \alpha_P= t \cdot u
\]
for some invertible polynomial $u \in \kk[t]/t^{n+1}$ (that is, $u$ has some non-zero constant term). It follows that $\alpha_P$ is a generator of $\Hom(P,P)\cong \kk[t]/t^{n+1}$. This is what we needed to check for $\FF$ to be a weak spherification functor.
\end{proof}

\begin{remark}
The only reason, why we prefer the left adjoint of $\FF$, is that the triangle  $\LL\FF \to \id \to \CC'$  appears in the geometric situation of \autoref{prop:geometric:spherification} which is our point of departure.
There is an analogous definition using the right adjoint of $\RR$, which yields similar statements.
\end{remark}

\begin{definition}
A \emph{spherification functor} $\FF \colon \cS \to \cT$ for a $\IP^n[k]$-like object $P$ is a weak spherification functor for $P$ such that $\FF P$ is a $(nk+k-1)$-Calabi--Yau object in $\cT$.
\end{definition}

\begin{remark}
Let $\FF \colon \cS \to \cT$ be a spherification functor for a $\IP^n[k]$-like object $P$.
Then by \autoref{lem:weak:spherification}, we have that $\FF P$ is $(nk+k-1)$-spherelike, and by definition of spherification functor that $\FF P$ is $(nk+k-1)$-Calabi--Yau.
Therefore $\FF P$ is a $(nk+k-1)$-spherical object in $\cT$.
\end{remark}

\begin{remark}
From now on, we will always assume that $n \ge2$ for all our $\IP^n[k]$-like objects. In other words, we exclude weak spherification of $k$-spherelike objects. This is coherent with our two main results \autoref{main:spherification} and \autoref{main:relations}.

We do not make any general assumptions on the value of $k$, except in \autoref{prop:commutativity} and therefore in \autoref{main:relations}; see \autoref{rem:k} for a few more comments on this. In the derived categories of smooth projective varieties, due to Serre duality, only $\IP^n[k]$-objects with $k>0$ can appear. 
In representation theory, however, $\IP^n[k]$-objects with $k\le 0$ can occur; see \autoref{rem:negativePk}.
The proof of \autoref{prop:pushforward}, which is an important ingredient of our proof of \autoref{main:relations}, would have been shorter (but not really conceptually easier) if we restricted ourselves to the case $k>0$, but we decided to go with the greater generality.

We note in passing that most of the following results still apply to the case $n=1$, that is,  $k$-spherelike objects (but in some cases the proofs must be slightly adapted) as long as $k>1$, but some break in the special case $k=1$. We decided it was not worth the complications that would arise when including always the special case $n=1$. The reason is that the question of relations of pairs of twists for $n=1$ is already settled by the results of \cite{Volkov}; see \autoref{rem:n=1}.

It is an idiosyncrasy of our definition that a weak spherification functor turns a $k$-spherelike object into a $(2k-1)$-spherelike object.
Note that, contrary to what one might expect from the name, the identity functor is not a weak spherification functor of spherelike objects as it has $\CC'=0$.

This remark applies also to spherification of spherical objects, mutatis mutandis.
\end{remark}

\begin{lemma}
\label{lem:spherification}
Let $\FF \colon \cS \to \cT$ be a spherelike functor with $\CC' \cong [k]$.
Then $\FF$ is a weak spherification functor for all $\IP^n[k]$-like objects $P$ such that the natural morphism $\alpha_P \colon P \to \CC'P \cong P[k]$ is a generator of $\Hom^*(P,P)$.

If $\FF$ is a spherical functor with $\CC' \cong [k]$, then $\FF$ is a spherification functor for all $\IP^n[k]$-objects $P$ with $\alpha_P\neq 0$.
\end{lemma}

\begin{proof}
The first part holds by the definition of a weak spherification functor.
Note that we are asking much more than necessary: $\CC'$ needs to be a shift only for the $\IP[k]$-like objects in question, a spherelike functor admits also a right adjoint which is not needed here.

Assume now that $\FF$ is a spherical functor, that is, $\CC$ is an equivalence  with $\CC^{-1} = \CC'$ and $\RR \cong \CC\LL[1]$.
We compute
\[
\begin{split}
\Hom^*(\FF P,\blank) & \cong \Hom^*(P, \RR \blank) \cong \Hom^*(\RR\blank, P[nk])^\vee \cong \\
& \cong \Hom^*(\CC\LL\blank[1],P[nk])^\vee \cong \Hom^*(\LL \blank, \CC^{-1} P[nk-1])^\vee \cong \\
& \cong \Hom^*(\blank, \FF P[nk+k-1])^\vee
\end{split}
\]
where we used adjunction, that $P$ is a $nk$-Calabi--Yau object, and $\CC' P \cong P[k]$.

Hence $\FF P$ is a $(nk+k-1)$-Calabi--Yau object, and therefore $\FF$ is a spherification functor.
\end{proof}

\subsection{The guiding example}

The following proposition is the motivation for our definition of a spherification functor, coming from \cite{HT06}.

Let $X$ be a smooth projective variety. 
Suppose there is a smooth family $\cX \to C$ over a smooth curve $C$ with 
distinguished fibre $j \colon X = \cX_0 \into \cX$, where $0 \in C$ is a closed point.
Note that the inclusion $j \colon X \to \cX$ gives rise to the push-forward functor $j_* \colon \cD(X) \to \cD(\cX)$ which admits both adjoints.

The family gives rise to the Atiyah class $\AA(E)$ for $E \in \cD(X)$ and the Kodaira-Spencer class $\kappa(\cX)$.
For further information on these classes, see \cite{HT10}.

We reformulate the following proposition of \cite{HT06} using the language of spherification functors.

\begin{proposition}[{\cite[Prop. 1.4]{HT06}}]
\label{prop:geometric:spherification}
Let $\cX \to C$ be a smooth family over a smooth curve $C$,
and let $j \colon X \to \cX$ be the inclusion of a fiber $X$.
Then $j_*$ is a spherification functor for all $\IP$-objects $P$ with $\AA(P)\cdot \kappa(\cX) \neq0$.
\end{proposition}

\begin{remark}
Note that the condition $\AA(P)\cdot \kappa(\cX) \neq0$ is exactly the one asked for in the definition of a weak spherification functor, as the triangle of the dual cotwist to $j^*$ is
\[
j^* j_* \to \id \xto{\AA(\blank)\cdot \kappa(\cX)} [2] \,.
\]
Also note that $j_*$ is a spherical functor, as $j \colon X \to \cX$ is the inclusion of a divisor.
So by \autoref{lem:spherification} and \autoref{rem:kneq0}, $j_*$ is indeed a spherification functor for all $\IP$-objects $P$ with $\AA(P)\cdot \kappa(\cX) \neq0$.

Note that associated to a $\IP$-object $P \in \cD(X)$, we have a $\IP$-functor $\FF_P = P \otimes \blank \colon \cD(\kk\hh\mod) \to \cD(X)$. 
Hence the proposition can be reformulated to the statement that the composition $j_* \circ \FF_P$ becomes a spherical functor.
So one might wonder, whether there is a condition generalising $\AA(\blank) \cdot \kappa(\cX) \neq 0$ (or the one of \autoref{lem:spherification} and \autoref{rem:kneq0}),
which guarantees that the composition of a $\IP$-functor and a spherical functor becomes spherical.
See \cite{MR19}, where this question is discussed for the spherical functor $j_*$.
\end{remark}

\begin{example}[{\cite[Ex. 1.5]{HT06}}]
\label{ex:geometric}
Let $X$ be a hyperk\"ahler variety of dimension $2n$. 
Then $X$ allows a $\IP^1$-family of complex structures, which fit together into an analytic manifold $\cX$, the \emph{twistor space} of $X$, see, for example, \cite{Huybrechts-habil} for a background.
So $X$ provides both a wealth of $\IP$-objects and a smooth (analytic) family $\cX \to \IP^1$.

Any non-trivial line bundle $L$ is a $\IP^n$-object in $\cD(X)$ with $\AA(L)\cdot \kappa(\cX) \neq0$. 
So $j_*$ is a spherification functor for any non-trivial line bundle $L$, that is, $j_* L$ is a spherical object in $\cD(\cX)$.

Note that $\cO_X$ cannot be spherified this way: it deforms to $\reg_{\cX}$.
Equivalently, $\AA(\cO_X)\cdot \kappa(\cX)$ vanishes, because of the splitting
\[
 j^* j_* \cO_X \cong \cO_X \oplus \cO_X[-1] \,.
\]
Hence, $j_*$ is not a spherification functor for $\cO_X$.

Also for any $\IP^n \subset X$, we get that $\cO_{\IP^n}$ is a $\IP^n$-object in $\cD(X)$ with $\AA(\cO_{\IP^n})\cdot \kappa(\cX) \neq0$. 
Again, $j_*$ is a spherification functor for such $\cO_{\IP^n}$, so $j_* \cO_{\IP^n}$ is a spherical object in $\cD(\cX)$.
\end{example}

\subsection{Existence of spherification functors}
\label{sec:existence}

The geometric situation of \autoref{prop:geometric:spherification} establishes already the existence of a spherification functor for many $\IP^n$-objects.
In this section, we construct spherification functors for more general $\IP^n[k]$-objects.

\begin{lemma}
\label{lem:formal:h:natural}
Let $A$ be a dg-algebra over $\kk$ and
let $h$ be a homogeneous element of even degree $k$  which is central and closed.
Then the cone of the multiplication by $h$
\[
A[-k] \xto{h\cdot} A \to B
\]
can be turned into a dg-algebra over $A$.
\end{lemma}

\begin{proof}
By the definition of the cone, as an $A$-module, $B$ can be written as $A \oplus A[1-k]$ with differential
\[
d_B = 
\begin{pmatrix}
d & h \\ 
0 & d
\end{pmatrix}
\]
where $d$ is the differential of $A$ (or the differential of $A[1-k]$ which differs by the one of $A$ by $(-1)^{1-k}=-1$).

We write $B = A[\eps]/\eps^2$ where $\eps$ is an element of degree $k-1$, which implicitly defines a graded multiplication 
\[
\begin{split}
(a_1 + \eps a_2 ) \cdot (a'_1 + \eps a'_2) 
=&\ a_1 a'_1 + \eps( (-1)^{(k-1)|a_1|} a_1a'_2 + a_2 a'_1) = \\
=&\ a_1 a'_1 + \eps( (-1)^{|a_1|} a_1a'_2 + a_2 a'_1)
\end{split}
\]
for homogeneous elements $a_i$ of degree $|a_i|$, where we use that $k$ is even.

The differential of $B$ is then
\[
d_B(a_1 + \eps a_2) = d(a_1) + h a_2 + (-1)^{k-1} \eps d(a_2) 
= d(a_1) + h a_2 - \eps d(a_2)
\]
again using that $k$ is even.

We check that with this differential $B$ becomes a dg-algebra, that is, that the Leibniz rule holds.
On the one hand, we have for homogeneous elements
\[
\begin{split}
&d_B((a_1+\eps a_2)(a'_1+\eps a'_2)) = \\
&\ =\ d_B( a_1 a'_1 + \eps( (-1)^{|a_1|} a_1a'_2 + a_2 a'_1) ) = \\
&\ =\ d( a_1 a'_1) + h( (-1)^{|a_1|}a_1a'_2 + a_2a'_1)  - \eps d((-1)^{|a_1|}a_1a'_2 + a_2a'_1) = \\
&\ =\ d(a_1) a'_1 + (-1)^{|a_1|} a_1 d(a'_1) + (-1)^{|a_1|} h a_1a'_2 + h a_2a'_1 - \\
&\ \phantom{=\ } - (-1)^{|a_1|} \eps d(a_1) a'_2 - \eps a_1 d(a'_2) - \eps d(a_2) a'_1 - (-1)^{|a_2|} \eps a_2 d(a'_1) \,.
\end{split}
\]
On the other hand, we have:
\[
\begin{split}
&d_B(a_1 + \eps a_2)(a'_1+\eps a'_2)+ 
 ((-1)^{|a_1|}a_1-(-1)^{|a_2|}\eps a_2)d_B(a'_1+\eps a'_2) = \\
&\ =\ (d(a_1) + h a_2 - \eps d(a_2))(a'_1+\eps a'_2) + \\
&\ \phantom{=\ } + ((-1)^{|a_1|}a_1-(-1)^{|a_2|}\eps a_2)(d(a'_1) + h a'_2 -\eps d(a'_2)) = \\
&\ =\ d(a_1)a'_1 + d(a_1)\eps a'_2 + ha_2a'_1 + h a_2 \eps a'_2 - \eps d(a_2) a'_1 + \\
&\ \phantom{=\ } + (-1)^{|a_1|}a_1 d(a'_1) + (-1)^{|a_1|}a_1 h a'_2 - (-1)^{|a_1|}a_1 \eps d(a'_2) - \\
&\ \phantom{=\ } - (-1)^{|a_2|}\eps a_2 d(a'_1) + (-1)^{|a_2|}\eps a_2 h a'_2
\end{split}
\]
which coincides with the previous computation, using that $h$ is central and the graded commutativity of $\eps$.
\end{proof}

Given a $\IP^n[k]$-object $P_i$,
we denote in the following by $t_i$ a non-zero element of $\End^k(P_i)$.

\begin{proposition}
\label{prop:formal:spherification}
Let $P_1,\ldots,P_m$ be $\IP^n[k]$-objects in some algebraic triangulated category with $k$ even.
Consider the dg-algebra $A = \End^\bullet(\bigoplus_i P_i)$, and suppose that there is a central and closed element $h = h_1 + \cdots + h_m \in A$ such that
\[
H^0(h_i) = t_i \in H^0(A) = \End^*\Bigl(\bigoplus_i P_i\Bigr) \,.
\]
Then there is a spherification functor $\FF \colon  \genby{P_1,\ldots,P_m} \to \cB$.
\end{proposition}

\begin{proof}
Note that $\genby{P_1,\ldots,P_m} \cong \cD(A)$.
We can apply \autoref{lem:formal:h:natural} to obtain the following triangle in $\cD(A)$
\begin{equation}\label{eq:B}
A[-k] \xto{h\cdot} A \to B
\end{equation}
where $B = A[\eps]/\eps^2$ can be turned into a dg-algebra with differential 
$d_B(a_1 + \eps a_2) = d(a_1) + h a_2 - \eps d(a_2)$.

The morphism $A \to B$ gives rise to 
\[
\FF = B \otimes_A \blank \colon \cD(A) \to \cD(B) \,. 
\]
Note that, by \eqref{eq:B}, $B$ is semi-free as a dg-module over $A$. Hence, we do not need to replace $B$ by a resolution in order to compute the derived tensor product $\FF$.
Furthermore, $B$ is perfect both as an $A$-module and as a $B$-module. 
Hence, $\FF$ has both adjoints; see the discussion at the beginning of \autoref{sec:spherification:definition}. Note that ${}^BB\cong B$, which implies that the right adjoint $\RR$ can be identified with the restriction functor ${}_A(\blank)$.

We compute first the cotwist of $\FF$ using the triangle
\[
\CC \to \id \to \RR\FF = {}_A(B \otimes_A \blank)  \,.
\]
In the corresponding triangle of $A$-$A$-bimodules $C \to A \to B$, we observe that the map $A\to B$ coincides with the one in \eqref{eq:B}, which is the map defining the $A$-algebra structure of $B$. Therefore we obtain that $C \cong A[-k]$, so $\CC \cong [-k]$.
Moreover, the morphism $\CC \to \id$ is just the multiplication with $h_i$. Dually, this holds also for the natural morphism $\id \to \CC'$.
In particular, by the first part of \autoref{lem:spherification}, $\FF$ is a weak spherification functor for the $P_i$.

Next we check that $\RR \cong \CC\LL[1] \cong \LL[1-k]$.
Again by semi-freeness of $B$ over $A$, the bimodule representing $\LL$ is ${}^AB\cong \Hom^\bullet_A(B,A)$ where the latter are the non-derived Hom-complexes of dg modules over $A$. We can calculate $\Hom^\bullet_A(B,A)$ by applying $\Hom^\bullet_A(\blank,A)$ to \eqref{eq:B}:
\[
\Hom^\bullet_A(A[-k],A) \from \Hom^\bullet_A(A,A) \from \Hom^\bullet_A(B,A) \,.
\]
From this we get that $\Hom^\bullet_A(B,A) \cong B[k-1]$ as $A$-modules.
Using the explicit description of $B = A[\eps]/\eps^2$ with $\deg(\eps)=k-1$, one can check that this holds also as $B$-modules.
Therefore, we have
\[
\LL = \Hom^\bullet_A(B,A) \otimes_B \blank \cong {}_A B \otimes_B (\blank)[k-1] \cong {}_A(\blank)[k-1] \cong \RR[k-1] \,.
\]
Hence $\FF$ is a spherical functor.

In particular, by the second part of \autoref{lem:spherification}, $\FF$ is a spherification functor of $P_1,\ldots,P_m$.
\end{proof}

The following is \autoref{main:spherification} of the introduction.

\begin{corollary}
\label{cor:formal:spherification}
Let $P_1,\ldots,P_m$ be $\IP^n[k]$-objects in some algebraic triangulated category,
such that $k\ge 2$ is even, $\gcd(k,nk/2)>1$ and $\Hom^*(P_i,P_j)$ is concentrated in degree $nk/2$ for $i\neq j$ (or zero).
Then there is a spherification functor $\FF \colon  \genby{P_1,\ldots,P_m} \to \cT$.
\end{corollary}

\begin{proof}
By slightly adapting the proof of \cite[Prop. 4.3]{Formality},
we get that $A = \End^*(\bigoplus_i P_i)$ is intrinsically formal.

Moreover, $h = t_1+\cdots+t_m$ is a central element
(note that, for $i\neq j$, the product of $h$ with elements in $\Hom^*(P_i,P_j)$ is zero for degree reasons).
So we can apply \autoref{prop:formal:spherification} to the dg-algebra $A$ with trivial differential, in order to conclude the proof.
\end{proof}

\begin{remark}
\label{rem:single:pnk}
In particular, we have that there is always a spherification functor for a single $\IP^n[k]$-object $P$ such that $k$ is even.
\end{remark}

\begin{remark}
Let $P_1$ and $P_2$ be two $\IP^n[k]$-objects, such that $k$ even and $\gcd(k,nk/2)>1$ (for example, if $n$ is even). In order to obtain a spherification functor for $P_1$ and $P_2$, it is enough to assume additionally that $\Hom^*(P_1,P_2)$ is concentrated in a single degree.
To see this, replace one of the two objects by a suitable shift, so we can ensure that $\Hom^*(P_1,P_2)$  (and by Serre duality $\Hom^*(P_2,P_1)$) is concentrated in degree $nk/2$.
\end{remark}

The following example starts from results taken from \cite{PS}, see also \cite[\S 6]{Formality}.

\begin{example}
\label{ex:hilb}
Let $Y$ be some smooth projective variety of even dimension $k$ and $S \in \cD(Y)$ a spherical object.

Then for $n>0$ we can consider the equivariant derived category $\cD_{\sym_n}(Y^n)$, where $S$ gives rise to two $\IP^n[k]$-objects $S^{+\{n\}}$ and $S^{-\{n\}}$ by linearising $S^{\boxtimes n}$ with the trivial or the sign representation of $\sym_n$, respectively.

If there are two spherical objects $S_1,S_2$ with $\Hom^*(S_1,S_2)$ in a single degree $d$, then we have that $\Hom^*\bigl(S_1^{\pm\{n\}},S_2^{\pm\{n\}}\bigr)$ is again concentrated in a single degree, namely $nd$.
In particular, there is a spherification functor for any of the pairs of $\IP$-objects of the form $\bigl(S_1^{\pm\{n\}},S_2^{\pm\{n\}}\bigr)$ by \autoref{cor:formal:spherification}.
\end{example}

In the situation of $\IP^n$-objects on hyperk\"ahler varieties,
we can often spherify using $j_* \colon \cD(X) \to \cD(\cX)$ of \autoref{prop:geometric:spherification}, or we can use the more abstract spherification of \autoref{cor:formal:spherification}.
There is a non-empty intersection, as we see in the following example.

\begin{example}
We specialise \autoref{ex:hilb}.
Recall that in case of $Y$ a surface, we get $\Phi \colon \cD_{\sym_n}(Y^n) \xto{\sim} \cD(\Hilb^n(Y))$.
In particular, if $Y$ is a K3 surface, $X = \Hilb^n(Y)$ is a hyperk\"ahler variety.
Note that if $S$ is a (non-trivial) line bundle on a K3 surface $Y$, then so is $\Phi\bigl(S^{\pm\{n\}}\bigr)$ on $X$, hence can be spherified using $j_* \colon \cD(X) \to \cD(\cX)$ or abstractly by \autoref{cor:formal:spherification}.

We noted in \autoref{ex:geometric} that $j_* \colon \cD(X) \to \cD(\cX)$ is not a spherification functor for $\cO_X = \Phi\bigl(\cO_Y^{+\{n\}}\bigr)$.
This problem can be circumvented easily: just precompose $j_*$ with a suitable autoequivalence, for example, tensoring with a non-trivial line bundle.
But by \autoref{rem:single:pnk}, we can also spherify $\cO_X$ directly, using the abstract spherification of \autoref{cor:formal:spherification}.

If $S = \cO_{C}$ for some rational curve on a K3 surface $Y$, then $S^{+\{n\}}$ becomes isomorphic to $\cO_{\IP^n}$ under the equivalence $\Phi^{-1}$, where $\IP^n \cong \Hilb^n(C) \into X$, see \cite[Prop. 6.6]{Formality}.
Hence, up to this equivalence, it can again be spherified either way.
But $S^{-\{n\}}$ will yield a more complicated object in $\cD(\Hilb^n(Y))$, see \cite[Prop. 6.7]{Formality} for $n=2$. For that reason it is not immediately clear, whether spherification using $j_*$ is possible. Still, we can spherify abstractly by \autoref{cor:formal:spherification}.
\end{example}

\begin{example}
\label{rem:negativePk}
As far as we know, there are no examples of $\IP^n[k]$-objects with $k<0$ in the literature.
But there are $k$-spherical objects with $k<0$ studied in representation theory. For example, \cite{Coelho} treats the triangulated category $\cT$ generated by a single $k$-spherical object.
Considering the symmetric power of this category $\sym^n \cT$ as introduced in \cite{GKsym},
one obtains two $\IP^n[k]$-objects there (like in the geometric situation of \autoref{ex:hilb}).
Note that \autoref{cor:formal:spherification} assumes $k\ge 2$, hence does not yield a spherification functor for these negative $\IP$-objects.
\end{example}

\subsection{Properties of spherification functors}

In the following, we write $\hom^*(A,B)$ for the dimension of the graded $\kk$-vector space $\Hom^*(A,B)$.

\begin{proposition}
\label{prop:pushforward}
Let $P_1,P_2$ be two $\IP^n[k]$-objects in some algebraic triangulated category $\cS$.
Let $\FF \colon \cS \to \cT$ be a weak spherification functor for $P_1$ and $P_2$.
\begin{enumerate}[label = (\roman*)]
\item \label{prop:pushforward:i} If $\hom^*(P_1,P_2) \geq 1$ then $\hom^*(\FF P_1,\FF P_2) \geq 2$.
\item \label{prop:pushforward:ii} $\FF P_1\cong \FF P_2$ implies $P_1\cong P_2$.
\end{enumerate}
\end{proposition}

\begin{proof}
Apply $\Hom^*(\blank,P_2)$ to the triangle $\LL\FF P_1 \to P_1 \to P_1[k]$ to get
\begin{equation}\label{eq:Homtriangle}
\Hom^*(\FF P_1,\FF P_2) \xleftarrow{\FF} \Hom^*(P_1,P_2) \xleftarrow{\alpha_{P_1}^*} \Hom^*(P_1,P_2)[-k] \,.
\end{equation}
Let us first prove \ref{prop:pushforward:i} in the case $k>0$. By \eqref{eq:Homtriangle}, we see that there is an injection of the minimal degree part of $\Hom^*(P_1,P_2)$ into $\Hom^*(\FF P_1,\FF P_2)$, and a surjection from $\Hom^*(\FF P_1,\FF P_2)$ to the maximal degree part of $\Hom^*(P_1,P_2)[-k]$. This gives the claim.

For $k<0$, we instead have an injection of the maximal degree part of $\Hom^*(P_1,P_2)$ into $\Hom^*(\FF P_1,\FF P_2)$, and a surjection from $\Hom^*(\FF P_1,\FF P_2)$ to the minimal degree part of $\Hom^*(P_1,P_2)[-k]$.

For $k=0$, we have that $\alpha_{P_i}\colon P_i\to P_i$ is a nilpotent endomorphism of degree zero for $i=1,2$, as $\Hom^*(P_i,P_i) \cong \kk[\alpha_{P_i}]/\alpha_{P_i}^{n+1}$; see \autoref{rem:alphagen}.
Hence, in the long exact cohomology sequence associated to \eqref{eq:Homtriangle}, the morphisms $\alpha_{P_i}^*\colon \Hom^j(P_1,P_2)\to \Hom^j(P_1,P_2)$ are nilpotent endomorphisms. In particular, they cannot be isomorphisms for any $j\in \IZ$. From this follows \ref{prop:pushforward:i} for the last remaining case, $k=0$.

We now prove \ref{prop:pushforward:ii}.
Since $\FF P_1\cong \FF P_2$, we have
\begin{equation}\label{eq:Hom12}
 \Hom^*(\FF P_1, \FF P_2)\cong  \Hom^*(\FF P_i, \FF P_i) \cong \kk \oplus \kk[-nk-k+1]
\end{equation}
where the last isomorphism is due to \autoref{lem:weak:spherification}.

Let us first assume that $k\neq 0$.
Reinspecting the arguments above involving \eqref{eq:Homtriangle} shows that
\begin{align*}
\mindeg \Hom^*(P_1,P_2)= \mindeg \Hom^*(\FF P_1,\FF P_2)\quad \text{for $k>0$,}\\
\maxdeg \Hom^*(P_1,P_2)= \maxdeg \Hom^*(\FF P_1,\FF P_2)\quad \text{for $k<0$,}
\end{align*}
where $\mindeg$ and $\maxdeg$ denote the minimal and maximal degrees, respectively, in which the graded vector spaces are non-zero.
Furthermore, we have $\mindeg \Hom^*(\FF P_1,\FF P_2)=0$ for $k>0$ and
$\maxdeg \Hom^*(\FF P_1,\FF P_2)=0$ for $k<0$; see \eqref{eq:Hom12}. Hence, in any case, we get an isomorphism of one-dimensional vector spaces $\FF\colon \Hom(P_1,P_2)\xto{\sim} \Hom(\FF P_1,\FF P_2)$ in degree $0$. In complete analogy, we have an isomorphism $\FF\colon \Hom(P_2,P_1)\xto{\sim} \Hom(\FF P_2,\FF P_1)$. Furthermore, inspecting the proof of \autoref{lem:weak:spherification} shows that there are isomorphisms $\FF\colon \Hom(P_i,P_i)\xto{\sim} \Hom(\FF P_i,\FF P_i)$ for $i=1,2$. Now, let $0\neq \phi\colon P_1\to P_2$ be some morphism. Then $\FF(\phi)\neq 0$, hence it is an isomorphism. Let $\psi\in \Hom(P_2,P_1)$ be the unique preimage of $\FF(\phi)^{-1}$ under $\FF\colon \Hom(P_2,P_1)\xto{\sim} \Hom(\FF P_2,\FF P_1)$. Then
\[
 \FF(\psi\circ \phi)=\FF(\psi)\circ \FF(\phi)=\id_{\FF P_1}\,.
\]
By the injectivity of $\FF\colon \Hom(P_1,P_1)\to \Hom(\FF P_1,\FF P_1)$, it follows that $\psi\circ \phi=\id_{P_1}$. The same way, we obtain  $\phi\circ \psi=\id_{P_2}$. Hence, $\phi$ and $\psi$ are mutually inverse isomorphisms.

Recall that, for $k=0$, the morphisms $\alpha_{P_i}^*\colon \Hom^j(P_1,P_2)\to \Hom^j(P_1,P_2)$ occuring in the long exact sequence associated to \eqref{eq:Homtriangle} cannot be isomorphisms for any $j\in \IZ$. Furthermore, $\Hom^*(\FF P_1, \FF P_2) \cong \kk \oplus \kk[1]$ by \eqref{eq:Hom12}. It follows that $\Hom^*(P_1,P_2)$ is concentrated in degree $0$ and $\FF\colon \Hom(P_1,P_2)\to \Hom(\FF P_1,\FF P_2)$ is surjective. Analogously, also $\FF\colon \Hom(P_2,P_1)\to \Hom(\FF P_2,\FF P_1)$ is surjective. Similarly to the case $k\neq 0$, we now pick some $\phi\in \Hom(P_1,P_2)$ such that $\FF(\phi)$ is an isomorphism, and some $\psi\in \Hom(P_2,P_1)$ with $\FF(\psi)=\FF(\phi)^{-1}$. Then $\FF(\psi\circ \phi)=\id_{\FF_{P_1}}$. Inspecting the proof of
\autoref{lem:weak:spherification}, we see that the kernel of
\[
 \FF\colon \Hom(P_1,P_1)\to \Hom(\FF P_1,\FF P_1)
\]
is, under the isomorphism $\Hom(P_1,P_1)\cong \kk[\alpha_{P_1}]/\alpha_{P_1}^{n+1}$, the maximal ideal $(\alpha_{P_1})$. It follows that $\psi\circ \phi$, not being contained in this kernel, is an automorphism of $P_1$. Analogously, $\phi\circ \psi$ is an automorphism of $P_2$. Hence, $\phi$ and $\psi$ are isomorphisms.
\end{proof}

\begin{example}
\label{ex:pushforward}
Let $\FF \colon \cS \to \cT$ be a weak spherification functor for two $\IP^n[k]$-objects $P_1$ and $P_2$ with $k\neq0$.
If $\Hom^*(P_1,P_2) \cong \kk^m[d]$, then we have $\Hom^*(\FF P_1, \FF P_2) \cong \kk^m[d] \oplus \kk^m[d-k+1]$.

To see this, we look at the triangle of the proof of \autoref{prop:pushforward}:
\[
\Hom^*(\FF P_1, \FF P_2) \from \Hom^*(P_1,P_2) \from \underbrace{\Hom^*(P_1,P_2[-k])}_{\mathclap{\cong \Hom^*(P_1,P_2)[-k] \cong \kk^m[d-k]}} \,.
\]
As there is no cancellation possible between the middle and the right hand term due to $k\neq0$, we arrive at $\Hom^*(\FF P_1, \FF P_2) \cong \kk[d] \oplus \kk[d-k+1]$.
\end{example}

The following statement generalises \cite[Prop. 4.7]{HT06} to arbitrary spherification functors.

\begin{proposition}
\label{prop:commutativity}
Let $P$ be a $\IP^n[k]$-object with $k\neq0,1$ in some algebraic triangulated category $\cS$.
Let $\FF \colon \cS \to \cT$ be a weak spherification functor for $P$.
Then the following diagram commutes:
\[
\begin{tikzcd}
\cS \ar[r, "\FF"] \ar[d, "\PPP_P"'] & \cT \ar[d, "\TTT_{\FF P}"] \\
\cS \ar[r, "\FF"'] & \cT
\end{tikzcd}
\]
\end{proposition}

\begin{proof}
First recall that $\TT = \TTT_{\FF P}$ fits into the triangle
\[
\FF P\otimes \Hom^*(\FF P,\blank) \xto{\mathrm{ev}_{\FF P}} \id \to \TT \,.
\]
Precomposing with $\FF$ we obtain
\begin{equation}
\label{eq:Tj}
 \FF P\otimes \Hom^*(\FF P,\FF\blank) \xto{\mathrm{ev}_{\FF P}\FF} \FF \to \TT \FF.
\end{equation}

On the other hand, apply $\FF$ to the triangles defining the $\IP$-twist $\PPP_P$ to get
\begin{equation}
\label{eq:jP}
\begin{tikzcd}[row sep=0.5ex,column sep={6em,between origins}]
\FF P[-k]\otimes \Hom^*(P,\blank)
\ar[rd, "\FF H"] \\
& \FF P\otimes \Hom^*(P,\blank)
\ar[rd] \ar[dd, "{\FF}\mathrm{ev}_P"]\\
&& \FF \Cone_H(\blank) \ar[ld, dashed]\\
& \FF \ar[ld] \\
\hphantom{\Hom^*(P,\blank)\otimes} \FF \PPP_P
\end{tikzcd}
\end{equation}
where we have $\FF H = \FF(\id \otimes t - t \otimes \id) = \id \otimes t$, as the second term cancels for degree reasons (we recall that we always assume that $n>1$). Here, $t$ is a generator of $\Hom^*(P,P)$ in degree $k$. By \autoref{rem:alphagen}, we can assume that $t=\alpha_P$.
Hence, $t$ fits also in the triangle of the dual cotwist for $P$:
\begin{equation}\label{eq:LFtriangle}
\LL \FF P \xto{\eps} P \xto{t} P[k]
\end{equation}
from which we get
\[
\begin{tikzcd}
\FF P[-k]\otimes \Hom^*(P,\blank) \ar[r, "\id \otimes t"] & \FF P\otimes \Hom^*(P,\blank) \ar[r, "\id \otimes \eps^*"] \ar{dr}{c} & \FF P\otimes \Hom^*(\LL \FF P,\blank) \ar{d}{\id\otimes \mathrm{adj}}[swap]{\cong} \\
&& \FF P\otimes\Hom^*(\FF P,\FF\blank)
\end{tikzcd}
\]
Let $\eta \colon \id \to \FF\LL$ be the unit of the adjunction, so we can write the adjunction isomorphism as 
\[
\mathrm{adj}=(\eta\FF)^* \circ \FF \colon \Hom^*(\LL \FF P,\blank) \xto{\sim} \Hom^*(\FF P,\FF\blank)
\]
Then we get
\[
\mathrm{adj} \circ \eps^* = (\eta\FF)^* \circ \FF \circ \eps^* = (\eta\FF)^* \circ (\FF\eps)^* \circ \FF = ( \FF\eps \circ \eta \FF)^* \circ \FF = \FF
\]
In particular, the composition $c$ in the above diagram is $c=\id \otimes \FF$, and the upper triangle in \eqref{eq:jP} is isomorphic to the triangle
\[
\FF P[-k]\otimes \Hom^*(P,\blank) \xto{t \otimes \id} \FF P\otimes \Hom^*(P,\blank) \xto{\FF \otimes \id} \Hom^*(\FF P,\FF\blank)\otimes \FF P
\]
Hence the right hand side of \eqref{eq:jP} has become
\[
\begin{tikzcd}[row sep=0.5ex,column sep={7em,between origins}]
\FF P\otimes \Hom^*(P,\blank)
\ar[rd, "\id\otimes \FF"] \ar[dd, "{\FF}\mathrm{ev}_P"]\\
&\FF P\otimes \Hom^*(\FF P,\FF\blank)  \ar[ld, dashed] \\
\FF 
\end{tikzcd}
\]

Note that one possible choice for a dashed arrow making the diagram commute is $\mathrm{ev}_{\FF P}\FF$, which is exactly the morphism whose cone is $\TT \FF$; see \eqref{eq:Tj}.
Recall \autoref{rem:bimodtwist}, which tells us that both diagrams, \eqref{eq:Tj} and \eqref{eq:jP}, defining the functors that we want to compare, are actually induced by diagrams of $\cB$-$\cA$-bimodules. Hence, 
to conclude $\FF\PP_P \cong \TT_{\FF P} \FF$, it suffices to show that the dashed morphism in the diagram corresponding to \eqref{eq:jP} in the derived category of bimodules is unique. Concretely, let $\FF\cong M\otimes\blank$, in particular $\FF P\cong M\otimes P$. Then, tensoring \eqref{eq:Ptwistbimod} on the left with $M$ gives the diagram
\[
\begin{tikzcd}[row sep=0.5ex,column sep={6em,between origins}]
\FF P \otimes {}^\cA P [-k]
\ar[rd] \\
& \FF P \otimes {}^\cA P
\ar[rd] \ar[dd]\\
&& M \otimes \Cone_H \ar[ld, dashed, "\phi"]\\
& M \ar[ld] \\
M\otimes X 
\end{tikzcd}
\]
in the derived category of $\cB$-$\cA$-bimodules where $M\otimes X$ represents the functor $\FF \PP_P$. 
By the long exact $\Hom(\blank,M)$-sequence associated to the upper exact triangle, for the 
uniqueness of the arrow $\phi$ making the triangle commute, it suffices to check the vanishing
\[
\Hom(\FF P \otimes {}^\cA P [1-k], M) = 0.
\]

By adjunction on the level of bimodules (see \cite[\S 2.2]{ALdg}) we have that
\[
\Hom_{\cD(\cB-\cA)}(\FF P \otimes {}^\cA P [1-k], M) \cong \Hom^{k-1}_{\cD(\cA-\cA)}(\LL \FF P \otimes {}^\cA P, \cA)
\]
where $\cA$ denotes the diagonal bimodule.
Applying 
$\Hom^*(\blank \otimes {}^\cA P,\cA)$ 
to the triangle \eqref{eq:LFtriangle} gives 
\[
\Hom^*(P \otimes {}^\cA P[k], \cA)
\to
\Hom^*(P \otimes {}^\cA P, \cA)
\to
\Hom^*(\LL \FF  P \otimes {}^\cA P, \cA)\,.
\]
Due to adjunction, the first two terms of this triangle are isomorphic to $\Hom^*(P,P)[-k]$ and $\Hom^*(P,P)$, respectively. Hence, the relevant part of the long exact sequence associated to this triangle is
\[
\begin{tikzcd}[column sep=2ex]
\Hom^{-1}(P , P) \ar[r]
& \Hom^{k-1}(P , P) \ar[r]  \ar[draw=none]{d}[name=Z, anchor=center]{}
& \Hom^{k-1}(\LL \FF P \otimes {}^\cA P,\cA)
\ar[rounded corners,
            to path={ -- ([xshift=2ex]\tikztostart.east)
                      |- (Z.center) \tikztonodes
                      -| ([xshift=-2ex]\tikztotarget.west)
                      -- (\tikztotarget)}]{dll}[at end]{} \\[-4ex]
\Hom^{0}(P , P) \ar[r] & \Hom^{k}(P ,P)
\end{tikzcd}
\]
where the first and the last arrow are precomposition with $t\colon P\to P[k]$. Hence, the last arrow $\Hom^{0}(P , P) \to \Hom^{k}(P ,P)$ is an isomorphism of one-dimensional vector spaces, by the assumption $k\neq 0$. For $k\notin \{-1,1\}$, the term $\Hom^{k-1}(P , P)$ vanishes. For $k=-1$, the multiplication by $t$ map $\Hom^{-1}(P , P) \to \Hom^{k-1}(P ,P)$ is an isomorphism of one-dimensional vector spaces. In summary, we proved the desired vanishing $\Hom^{k-1}(\LL \FF  P \otimes {}^\cA P, \cA)=0$ for all $k\neq 0,1$.
\end{proof}

\begin{remark}\label{rem:k}
For $k=0,1$ the last step of the argument in the proof of \autoref{prop:formal:spherification} breaks. 
Still, we believe that the statement holds also for $k=0$, which would then also give \autoref{thm:main} for $k=0$. In the case of $k=1$, we are a bit more sceptical due to \autoref{rem:11}. 

The lift $\phi$ is not unique anymore for $k=0,1$. However, it might be possible to prove by using methods similar to those of \cite{ALunique}, that its cone is unique, which is a convolution of the three-term complex
\[
\FF P \otimes {}^\cA P [-k]\to \FF P \otimes {}^\cA P \to M\,.
\]
Note that we cannot directly infer uniqueness of the convolution, as \emph{loc.\ cit}. is about the convolution of the three-term complex before applying $\FF$.
\end{remark}

\begin{remark}
\label{rem:inverse}
If $\FF \colon \cS \to \cT$ is a spherification functor for $P$, then we get also
\[
\TTinv_{\FF P} \FF \cong \FF \PPinv_P
\]
by precomposing $\FF \PPP_P \cong \TTT_{\FF P} \FF$ with $\PPinv_P$ and postcomposing it with $\TTinv_{\FF P}$.
\end{remark}

\section{Returning from spherical objects to $\IP$-objects}

The following proposition is part of a far more general statement about spherical twists associated to \emph{spherical sequences}.
To our knowledge, spherical sequences were first defined and studied in \cite{Efimov}, but they were rediscovered independently as \emph{exceptional cycles} in \cite{BPP}. Anyway, we will stick to the special case of spherical objects.

\begin{proposition}[{\cite[Thm. 2.7]{Volkov}}]
\label{prop:spherical:norels}
Let $S_1, S_2$ be two $k$-spherical objects in some algebraic triangulated category which are not isomorphic up to shift.
If $\hom^*(S_1,S_2) \geq2$, then there are no relations among the associated spherical twists.
\end{proposition}

\begin{remark}
In \cite[Thm. 1.2]{Keating}, a similar statement to \autoref{prop:spherical:norels} is shown in the context of Lagrangian spheres inside a symplectic manifold (giving rise to spherical objects in a derived Fukaya category).

In \cite[Thm. 1.2]{Anschuetz}, the statement was shown in the special case of a $\tilde A_1$-configurations of spherical objects in K3 categories.
\end{remark}

The following gives \autoref{main:relations} of the introduction.

\begin{theorem}
\label{thm:main}
Let $P_1,P_2$ be two $\IP^n[k]$-objects with $k\neq0,1$ in some algebraic triangulated category $\cS$ which are not isomorphic up to shift.
Let $\PP_1, \PP_2$ be the associated $\IP$-twists.
Suppose that there is a spherification functor $\FF \colon \cS \to \cT$ for $P_1$ and $P_2$.
If $\Hom^*(P_1,P_2)\neq 0$, then there are no relations between the $\IP$-twists, that is, $\genby{\PP_1,\PP_2} \cong F_2$.
\end{theorem}

\begin{proof}
The objects $S_1 = \FF P_1$ and $S_2 = \FF P_2$ are spherical by definition,
and by \autoref{prop:pushforward}, we have that $\hom^*(S_1,S_2) \geq 2$, and that $S_1,S_2$ are not isomorphic up to shift.
So we find that $\genby{\TTT_1,\TTT_2} \cong F_2$ by \autoref{prop:spherical:norels} where 
 $\TTT_1:=\TTT_{S_1}$ and $\TTT_2:=\TTT_{S_2}$ denote the associated spherical twists.

We recall the key idea of the proof of \autoref{prop:spherical:norels}; see the proof of \cite[Cor. 3.4]{Volkov}.
We denote by $\cO_\TT$ the union of the orbits of $S_1$ and $S_2$ under the action of $\genby{\TTT_1,\TTT_2}$ on $\cB$, that is,
\[
\cO_\TT = \genby{\TTT_1,\TTT_2} \cdot \{S_1,S_2\} \,.
\]
We define two subsets of $\cO_\TT$:
\[
\begin{split}
\cX  &= \{ X \in \cO_\TT \mid \hom^*(S_2,X) > \frac{\hom^*(S_1,S_2)}2 \cdot \hom^*(S_1,X) \} \,,\\
\cX' &= \{ X \in \cO_\TT \mid \hom^*(S_1,X) > \frac{\hom^*(S_1,S_2)}2 \cdot \hom^*(S_2,X) \} \,.
\end{split}
\]
As $\hom^*(S_1,S_2)\geq2$, the two sets do not intersect.

Now Y. Volkov shows that $\TTT_1 S_2 \in \cX$ (and analogously, $\TTT_2 S_1 \in \cX'$), so both sets are non-empty.
Finally, he shows that $\TT_{\!1}^m \cdot \cX' \subseteq \cX$ and $\TT_{\!2}^m \cdot \cX \subseteq \cX'$ for $m\neq0$.
Hence, one can conclude by the Ping-Pong Lemma, that $\TTT_1$ and $\TTT_2$ generate the free group $F_2$.

To obtain the statement about $\IP^n[k]$-objects, we bring \autoref{prop:commutativity} into play.
Define for $\cO_\PP = \genby{\PP_1,\PP_2} \cdot \{P_1,P_2\}$ the subsets
\[
\begin{split}
\cY  &= \{ Y \in \cO_\PP \mid \FF Y \in \cX \} \,, \\
\cY' &= \{ Y \in \cO_\PP \mid \FF Y \in \cX' \} \,.
\end{split}
\]
We note that both are non-empty, as $\PP_1 P_2 \in \cY$ and $\PP_2 P_1 \in \cY'$, using \autoref{prop:commutativity}. By the same proposition and \autoref{rem:inverse}, $\PP_1^m \cdot \cY'\subset \cY$ and $\PP_2^m \cdot \cY\subset \cY'$ for $m\neq 0$. Hence, we can now conclude that the Ping-Pong Lemma applies here, so
$\PP_1$ and $\PP_2$ generate the free group $F_2$.
\end{proof}

\begin{remark}\label{rem:n=1}
 As mentioned in the introduction, our assumption that $n\ge 2$ is not logically necessary for \autoref{thm:main}. Indeed, for $n=1$, there are never relations between $\PPP_{P_1}$ and $\PPP_{P_2}$ if $\Hom^*(P_1,P_2)\neq 0$, independently of the existence of a spherification functor. The reason is that by \cite{Volkov} the spherical twists $\TTT_{P_1}$ and $\TTT_{P_2}$ either generate the free group $F_2$ or the braid group $A_2$. In either case, its squares $\PPP_{P_i}=\TTT_{P_i}^{\ 2}$ do not satisfy any relations.  
\end{remark}

\section{Twists along orthogonal $\IP$-objects}
\label{sec:orthogonal}

\begin{proposition}
 Let $P_1, P_2$ be two $\IP^n[k]$-objects in some algebraic triangulated category with $\Hom^*(P_1,P_2)=0$. Then the following holds:
\begin{itemize}
\item The associated $\IP$-twists $\PP_1$ and $\PP_2$ commute.
\item If $(n,k)\neq (1,1)$, the two $\IP$-twists span a free abelian group: 
\[
\langle \PP_1,\PP_2\rangle\cong \IZ^2\,.
\]
\end{itemize}
\end{proposition}

\begin{proof}
This is proved in \cite[Cor.\ 2.5]{Krug-HK}, but there it is formulated only for $\IP^n$-objects in (equivariant) derived categories of smooth projective varieties. The general proof is exactly the same. However, we reproduce it here, as it is not long, and we want to make clear where the assumption $(n,k)\neq (1,1)$ is needed.

By \cite[Prop.\ 2.1]{Krug-HK}, which summarises some results from \cite{Add}, we have
\begin{align}
&\PP_i(P_j)=P_j \quad\text{for $i\neq j$,} \label{eq:P1}\\
 &\PP_i(P_i)=P_i[-(n+1)k+2]\,. \label{eq:P2}
\end{align}
By \eqref{eq:P1} together with \cite[Lem.\ 2.4(ii)]{Krug-HK}, we have
\[
 \PP_1=\PPP_{P_1}=\PP_{\PP_2(P_1)}=\PP_2\PP_1\PP_2^{-1}
\]
which proves the first part of the assertion.

For the second part, note that, due to the commutativity, we can write every $g\in \langle \PP_1,\PP_2\rangle$ in the form $g=\PP_1^{\nu_1}\PP_2^{\nu_2}$. By \eqref{eq:P1} and \eqref{eq:P2}, we have
\[
g(P_i)=P_i\bigl[\nu_i(-(n+1)k+2)\bigr]\quad\text{for $i=1,2$.}
\]
In particular, only for $\nu_1=0=\nu_2$ it happens that $g=\id$, as for $(n,k)\neq(1,1)$ we find that $-(n+1)k+2 \neq 0$.
\end{proof}

\begin{remark}\label{rem:11}
The assumption $(n,k)\neq (1,1)$ is not only needed for the proof, but actually there are counterexamples to the statement for $(n,k)=(1,1)$. In other words, in the case when we are speaking about $1$-spherical objects.

The following can be found in \cite[\S 3.4]{ST} and \cite[Ex.\ 8.25]{Huy} as part of an example of a $A_3$-sequence of $1$-spherical objects where the associated braid group action is not faithful. Let $C$ be an elliptic curve, and let $x,y\in C$ be two points such that $\reg_C(x-y)$ is a $2$-torsion line bundle. Let $P_1=\reg_x$ and $P_2=\reg_y$ be the associated skyscraper sheaves. They are orthogonal $1$-spherical objects with associated spherical twists
\[
 \TTT_1=\reg_C(x)\otimes\blank\quad\text{and}\quad \TTT_2= \reg_C(y)\otimes \blank\,.
\]
As the $\IP$-twists are the squares of the spherical twists, $\PP_i=\TT_{\!i}^2$, we get the relation
\[
\PP_1\PP_2^{-1}=\reg_C(2x-2y)\otimes\blank=\reg_C\otimes\blank=\id
\]
which means that the two $\IP$-twists are equal in this case.
\end{remark}

\section{Open questions and speculation}

Let us finish by asking some open questions related to our results. Recall that we proved the absence of relations between $\IP$-twists associated to non-orthogonal pairs of $\IP$-objects under the assumption that a spherification functor for both objects exists, and provided the existence of spherification functors for many, but not all, pairs of $\IP$-objects.

\begin{question}\label{Q1}
Given a pair $P_1,P_2$ of $\IP^n[k]$-objects in some algebraic triangulated category $\cS$, is there always a spherification functor $\FF\colon \langle P_1,P_2\rangle\to \cT$ for $P_1$ and $P_2$? 
More generally, one may ask the same question for a collection $P_1,\ldots,P_m$ of $\IP^n[k]$-objects.
\end{question}

It seems to us, that without the assumption of \autoref{prop:formal:spherification} that there is a central element of $A=\End^\bullet(P_1\oplus P_2)$ lifting $t_1+t_2\in \End^*(P_1\oplus P_2)$, then one has to come up with a construction really different from the one used there 
(\emph{provided} it is even true that there is a spherification functor in this case). 
To explain the problem, let us rephrase the idea of our construction of a spherification functor $\FF$ of \autoref{prop:formal:spherification}. We want to have a functor $\FF$ whose associated comonad $\LL \FF$ is the cone of the morphism of functors $h\colon [-k]\to \id$ in $\cD(A)$. However, if $h$ is not central, there is not even a morphism of functors $h\colon [-k]\to \id$ as, for a general dg-module $M$ over $A$, multiplication by $h$ will not be an $A$-linear map $M[-k]\to M$.      

\begin{question}\label{Q2}
Given a pair $P_1,P_2$ of non-orthogonal $\IP^n[k]$-objects in some algebraic triangulated category $\cS$ which are not isomorphic up to shift, is it always true that there are no relations between the associated twists? 
\end{question}

By \autoref{thm:main}, a positive answer to \autoref{Q1} would imply a positive answer to \autoref{Q2}. 
If  the answer to \autoref{Q1} should be negative, one could still try to prove a positive answer to \autoref{Q2} by hand. This seems feasible, but much more work than just pulling-back Volkov's result about the absence of relations between spherical twists (or rather its proof via the ping-pong lemma) along a spherification functor. 

On the other hand, the authors would not be completely surprised if there should be examples of $\IP^n[k]$-objects $P_1,P_2$ where the associated $\IP$-twists $\PP_1, \PP_2$ satisfy some relation when $t_1+t_2$ is not a central element of $\End^*(P_1\oplus P_2)$, in which case the answer to \autoref{Q2} would be negative. The reason is that the action of $t_1$ and $t_2$ on $\Hom^*(P_1,P_2)$ and $\Hom^*(P_2,P_1)$ comes into play when we compute $\PP_1 P_2$ and $\PP_2 P_1$; see \autoref{prop:HTtwist}. 

Another natural question is about relations (or the lack thereof) between more than just two $\IP$-twists. However, there the answer is not even known for spherical twists, so it would be natural to first try to extend Volkov's results to more than just two spherical twists.

A last question is if and how our results generalise from $\IP$-objects to $\IP$-functors. Again, it seems like a sensible approach to first try to generalise Volkov's results from spherical objects to spherical functors.

\bibliographystyle{alpha}
\bibliography{biblio}

\end{document}